\documentclass[11pt]{amsart}

\usepackage{amsmath, amssymb, amscd, cancel, graphicx, soul, stmaryrd}

\usepackage{mathdots}

\usepackage[all]{xy}
\usepackage{multirow}
\usepackage{longtable}
\usepackage{array}

\newtheorem{thm}{Theorem}[section]

\newtheorem{prop}[thm]{Proposition}

\def\s{{\mathfrak s}}
\def\t{{\mathfrak t}}

\def\cT{{\mathcal T}}
\def\cX{{\mathcal X}}
\def\bZ{{\mathbb Z}}

\def\Char{{\mathrm{Char}}}

\def\spc{{\mathrm{spin^c}}}
\def\Spc{{\mathrm{Spin^c}}}
\def\Short{{\mathrm{Short}}}

\def\del{{\partial}}

\def\im{{\text{im}}}
\def\mod{{\textup{mod} \;}}
\def\rk{{\mathrm{rk}}}

\newcommand{\into}{\hookrightarrow}

\begin{document}

\title[A note on applications of the $d$-invariant and Donaldson's theorem]%
{A note on applications of the $d$-invariant and Donaldson's theorem}

\author[Joshua Evan Greene]{Joshua Evan Greene}

\address{Department of Mathematics, Boston College\\ Chestnut Hill, MA 02467}

\email{joshua.greene@bc.edu}

\maketitle

\begin{abstract}
This note contains two remarks about the application of the $d$-invariant in Heegaard Floer homology and Donaldson's diagonalization theorem to knot theory.
The first is the equivalence of two obstructions they give to a 2-bridge knot being smoothly slice.  The second carries out a suggestion by Stefan Friedl to replace the use of Heegaard Floer homology by Donaldson's theorem in the proof of the main result of \cite{greene:2013-2} concerning Conway mutation of alternating links.
\end{abstract}

{\it\hfill Dedicated to the memory of Tim Cochran}


\section{Introduction.}

Donaldson's diagonalization theorem and Heegaard Floer homology have led to great success in knot theory.  In this note, we focus on two specific applications of these tools to knot concordance and mutation that appear in the literature.  We show how in both cases they can be used interchangeably towards the same end.  The moral is that for the applications considered herein, the $d$-invariant in Heegaard Floer homology simply repackages the information already carried by Donaldson's theorem.   This is not unexpected in light of the close relationship between them (see \cite[Section 9]{os:absgr}).

We first briefly recall both tools at work. Donaldson's theorem asserts that if $Z$ is a closed, oriented, smooth 4-manifold whose intersection pairing $Q_Z$ is definite, then $H_2(Z;\bZ)/\mathrm{Tors}$ admits an orthonormal basis with respect to $Q_Z$ \cite[Theorem 1]{d:thma}.  The $d$-invariant is a highly useful invariant defined by Ozsv\'ath and Szab\'o in Heegaard Floer homology.  It is modeled on the $h$-invariant defined by Fr\o yshov in Seiberg Witten Floer homology.  It assigns a rational number $d(Y,\t)$ to a closed, oriented 3-manifold $Y$ equipped with a torsion $\spc$ structure $\t$.

To set the stage for the main results, we recall how both of these tools can be used in order to prove that the pretzel knot $P(-3,5,7)$ is not smoothly slice.  By contrast, this knot is topologically slice, since it has trivial Alexander polynomial \cite{fq:book}.  The existence of topologically slice knots that are not smoothly slice was a sensational early application of Donaldson's and Freedman's work.  According to the paper of Cochran and Gompf, the line of argument using Donaldson's theorem is due to Casson \cite[\S1]{cg:1988}.  The line of argument using Heegaard Floer homology is standard by now.  As we shall see, the two proofs are slight variations of one another.  

The starting point for both proofs is the following observation.  If a knot $K$ bounds a disk $D$ smoothly and properly embedded in $D^4$, then the double cover of $D^4$ branched along $D$ is a smooth, compact 4-manifold $\Sigma(D)$ with the $\bZ / 2 \bZ$ (and hence rational) homology groups of a ball, and its boundary $\Sigma(K)$ is the double cover of $S^3$ branched along $K$ \cite[Lemma 2]{cg:cobordism}. For the case $K = P(-3,5,7)$, the manifold $\Sigma(K)$ is the Brieskorn sphere $-\Sigma(3,5,7)$.  This space bounds a smooth, compact 4-manifold $X$ obtained by plumbing disk-bundles over spheres and for which $(H_2(X;\bZ),Q_X)$ is isometric to the unimodular, definite lattice $D_{12}^+$ (see \cite[Section 3.2]{os:plumbed}).

Now, if $K = P(-3,5,7)$ were smoothly slice, then $\Sigma(K)$ would also bound a smooth rational homology ball $W$.  The union $Z = X \cup (-W)$ would then be a smooth, closed 4-manifold with $(H_2(Z;\bZ)/\mathrm{Tors},Q_Z)$ isometric to $D_{12}^+$, which does not admit an orthonormal basis, in violation of Donaldson's theorem.  Therefore, $K$ is not smoothly slice. Alternatively, following \cite[Section 3.2]{os:plumbed}, a calculation with $D_{12}^+$ shows that $d(Y,\t) = -2$, where $\t$ denotes the unique $\spc$ structure on $Y$.  Since $\t$ is the unique $\spc$ structure on $\Sigma(K)$, it extends over any smooth, compact 4-manifold filling $\Sigma(K)$.  On the other hand, if $\t$ extends to a $\spc$ structure on a rational homology ball $W$ that fills $Y$, then $d(Y,\t) = 0$ \cite[Proposition 9.9]{os:absgr}.  It follows once more that no rational homology ball fills $\Sigma(K)$, so $K$ is not smoothly slice.  

The two proofs that $P(-3,5,7)$ is not smoothly slice are both based on the existence of the 4-manifold $X$, properties of the $D_{12}^+$ lattice, and a suitably sensitive tool in smooth 4-manifold topology.  In fact, the result and both proofs generalize to any knot $K$ for which $\Sigma(K)$ bounds a 4-manifold with a positive definite intersection pairing not isometric to the Euclidean lattice $\bZ^n$.  The proof using Donaldson's theorem generalizes directly.  The proof using Heegaard Floer homology does as well, as the $d$-invariant of such a manifold is negative by \cite[Theorem 9.6]{os:absgr} and a theorem of Elkies \cite{elkies}.

In Section \ref{s: rational balls} we show that, in much the same way, two obstructions in the literature to a two-bridge knot being smoothly slice, one from Donaldson's theorem and one from the $d$-invariant, are equivalent.  In Section \ref{s: mutation} we highlight a novel instance in which Donaldson's theorem can be used in place of Heegaard Floer homology.  This possibility was pointed out by Stefan Friedl.  The main result of \cite{greene:2013-2} asserts that if $L$ and $L'$ are alternating links, then $\Sigma(L) \approx \Sigma(L')$ if and only if $L$ and $L'$ are mutants.  Moreover, the $d$-invariant of $\Sigma(L)$ is a complete invariant of the mutation type of $L$.  The argument hinges on an expression for the $d$-invariant of $\Sigma(L)$ due to Ozsv\'ath and Szab\'o \cite[Theorem 3.4]{os:doublecover}.  The expression is defined in terms of a lattice $\Lambda(D)$ associated with an alternating link diagram $D$ of $L$.  It follows from the invariance of the $d$-invariant that if $D'$ is an alternating link diagram of $L'$ and $\Sigma(L) \approx \Sigma(L')$, then the formulas for the $d$-invariant derived from $\Lambda(D)$ and $\Lambda(D')$ are the same.  Friedl's suggestion was to show that these formulas are the same by an appeal to Donaldson's theorem instead of Heegaard Floer homology.  We carry out the details of this suggestion in Section \ref{s: mutation}.


\section*{Acknowledgments.}

I am pleased to thank Stefan Friedl for the excellent suggestion he made during my lecture on \cite{greene:2013-2} at the conference ``The topology of 3-dimensional manifolds" at CRM in Montreal on May 16, 2013.  Thanks as well to Brendan Owens for many enlightening discussions on these topics over the years.  This work was partially supported by NSF CAREER Award DMS-1455132 and an Alfred P. Sloan Research Fellowship.


\section{Sliceness of two-bridge knots.}
\label{s: rational balls}

Let $Y$ denote an oriented rational homology 3-sphere, and suppose that $Y$ bounds an oriented rational homology 4-ball $W$.  As remarked above, Ozsv\'ath and Szab\'o showed that the invariant $d(Y,\t)$ vanishes for any $\spc$ structure $\t$ on $Y$ that extends across $W$.  A lot of work has gone into using this fact as an obstruction: given a rational homology sphere $Y$, one attempts to argue that it is not the boundary of {\em any} rational homology ball.  For instance, if {\em all} of the correction terms $d(Y,\t)$ are non-zero, then one concludes that $Y$ does not bound a rational ball, as we did above in the case of $\Sigma(3,5,7)$.  Variations on this theme are carried out in \cite{greenejabuka:2011}, \cite{grs:conc}, \cite{jn:conc}, \cite{ana:thesis}, \cite{lisca:lens1,lisca:lens2}, and \cite{owensstrle:2012}.

Most applications of this idea require somewhat more: one seeks more {\em a priori} conditions on which $\spc$ structures on $Y$ could extend over a putative rational ball, and to combine these conditions with the vanishing of the correction terms.  Casson and Gordon observed that the image of the restriction mapping $r_H: H^2(W;\bZ) \to H^2(Y;\bZ)$ is a subgroup $\im(r_H)$ whose order is the square-root of $|H^2(Y;\bZ)|$, implying the latter value is a perfect square $m^2$ \cite[Lemma 3]{cg:cobordism}.  This observation holds at the level of $\Spc$ structures: the image of the restriction mapping $r_S: \Spc(W) \to \Spc(Y)$ forms a torsor over the subgroup $\im(r_H)$.  Furthermore, there is a conjugation action on $\Spc(W)$ and $\Spc(Y)$ which commutes with the restriction map.
Thus, in order for $Y$ to bound a rational ball, the $d$-invariant vanish on a conjugation-invariant subtorsor of $\Spc(Y)$ over a subgroup of order $m$.  

In the application to knot concordance, we assume that $Y = \Sigma(K)$ for some knot $K \subset S^3$.  This has the added feature that $Y$ is a $\bZ /2 \bZ$ homology sphere, and if $K$ is smoothly slice, then $Y$ is filled by the smooth $\bZ / 2 \bZ$ homology ball $W = \Sigma(D)$, where $D$ denotes the slice disk. There exists a first Chern class mapping $c_1 : \Spc(\cdot) \to H^2(\cdot;\bZ)$ for both $Y$ and $W$.  It is a torsor isomorphism since $H^2(\cdot;\bZ)$ has no 2-torsion, it commutes with the restriction map, and $c_1(\overline{\s}) = - c_1(\s)$ for conjugate $\spc$ structures $\s, \overline{\s} \in \Spc(W)$.  If $H^2(Y;\bZ)$ is a cyclic group of odd order $m^2$, then it follows that $\im(c_1 \circ r_S)$ is the unique cyclic subgroup of order $m$.
This identifies $T := c_1^{-1}(m \cdot H^2(Y;\bZ)) \subset \Spc(Y)$ with $\im(r_S)$.  Thus, in order to argue that $K$ is not smoothly slice, it suffices to show that the $d$-invariant of $Y$ does not vanish on $T$.

We shall use the following lattice-theoretic description of $T$. Express $Y$ by integral surgery along a framed link $L \subset S^3 = \del D^4$.  Attaching 2-handles to $D^4$ along $L$ produces a 4-manifold $X$ with $\del X = Y$, $H_1(X;\bZ) = 0$, and whose intersection pairing $\Lambda = (H_2(X;\bZ),Q_X)$ is presented by the linking matrix of $L$.  Now glue $X$ and the putative $\bZ / 2 \bZ$ homology ball $-W$ by an orientation-reversing diffeomorphism of their boundaries to produce a closed 4-manifold $Z$.  By Poincar\'e duality, the intersection pairing lattice $\Lambda' = (H_2(Z;\bZ), Q_Z)$ is unimodular and integral, and the inclusion $X \into Z$ induces an inclusion $\Lambda \into \Lambda'$.  Since $H_1(X;\bZ) = 0$, every $\spc$ structure on $Y$ extends across $X$, so a $\spc$ structure on $Y$ extends across $W$ if and only if it extends across $Z$.  Since $H^1(X;\bZ/2\bZ)$ and $H^1(Z;\bZ/2\bZ)$ vanish, the first Chern class mapping $c_1$ establishes one-to-one correspondences $\Spc(X) \leftrightarrow \Char(\Lambda)$ and $\Spc(Z) \leftrightarrow \Char(\Lambda')$, the sets of characteristic elements for these lattices \cite[pp.56-57]{gs:4mflds}.  Under the first correspondence, the restriction mapping $\Spc(X) \to \Spc(Y)$ corresponds to the mapping $\Char(\Lambda) \to \cX(\Lambda)$, where $\cX(\Lambda)$ denotes the set of equivalence classes of $\Char(\Lambda)$ modulo $2 \Lambda$.  Under the identification $\Spc(Y) \leftrightarrow \cX(\Lambda)$, the set of $\spc$ structures on $Y$ that extend over $Z$ is precisely $\Char(\Lambda') \, (\mod 2 \Lambda)$.  Specializing to the case that $H^2(Y;\bZ)$ is cyclic of order $m^2$, it follows that $T$ can be identified with $\Char(\Lambda') \, (\mod 2 \Lambda)$.

As an application of these ideas, suppose that the two-bridge knot $S(p,q)$ is smoothly slice.  Then $\Sigma(K) \approx L(p,q)$ bounds a $\bZ /2 \bZ$ homology ball $W$.  The previous discussion implies that $p = m^2$ is odd and there exists a unique subset $T \subset \Spc(L(p,q))$ of order $m$ which can extend over $W$. Moreover, we have the following condition:
\begin{equation}\label{condition one}
d(L(p,q),\t) = 0, \, \forall \, \t \in T.
\end{equation}
\noindent Variations on this condition appear in \cite{grs:conc, jn:conc}, where it and its enhancements get used in order to bound the concordance orders of some 2-bridge knots.

On the other hand, the lens space $L(p,q)$ bounds a positive definite plumbing manifold $X(p,q)$ with intersection pairing lattice $\Lambda_1$, and $L(p,p-q)$ bounds a positive definite plumbing manifold $X(p,p-q)$ with intersection pairing lattice $\Lambda_2$.  Both $X(p,q) \cup (-W)$ and $X(p,p-q) \cup W$ are smooth, closed, definite 4-manifolds, so by Donaldson's theorem they have diagonalizable intersection pairing lattices $\bZ^{r_1}$ and $\bZ^{r_2}$, respectively, where $r_i = \rk(\Lambda_i)$.  We obtain the following condition:
\begin{equation}\label{condition two}
\Lambda_i \into \bZ^{r_i}, i = 1,2.
\end{equation}
\noindent This condition appears in the work \cite{lisca:lens1,lisca:lens2}.  By a remarkable combinatorial argument, Lisca showed that this condition is also sufficient.  Moreover, he used it to determine the concordance orders of {\em all} two-bridge knots.

Thus, \eqref{condition one} and \eqref{condition two} are two {\em a priori} different necessary conditions for a lens space to bound a rational ball.  In light of Lisca's result, it is clear that \eqref{condition two} is at least as strong as \eqref{condition one}.  It stands to wonder whether \eqref{condition one} could have been used towards the same conclusion.
In fact, as we now argue, conditions \eqref{condition one} and \eqref{condition two} are equivalent (independently of Lisca's result).

\begin{prop}
\label{p: equivalence}
Conditions \eqref{condition one} and \eqref{condition two} on a lens space are equivalent.
\end{prop}

\begin{proof}
\eqref{condition one} $\implies$ \eqref{condition two}: As discussed, $T$ can be identified with $\Char(\Lambda_i') \, (\mod 2 \Lambda_i)$, where $\Lambda_i'$ is the unique unimodular lattice with $\Lambda_i \subset \Lambda_i' \subset \Lambda_i^*$.  By \eqref{condition one}, $d$ vanishes on this subset, so the $d$-invariant of $\Lambda_i'$ is 0.  By Elkies's theorem \cite{elkies}, it follows that $\Lambda_i' \approx \bZ^{r_i}$, so condition \eqref{condition two} holds.

\noindent \eqref{condition one} $\implies$ \eqref{condition two}: Since the $d$-invariant of $\bZ^{r_i}$ is zero, it follows from \eqref{condition two} that $d(L(p,q),\t)$ and $d(L(p,p-q),\t)$ are non-negative for all $\t \in \cT$.  On the other hand, $d(L(p,q),\t) = -d(L(p,p-q),\t)$, so both values vanish, and \eqref{condition one} holds.
\end{proof}

Observe that the statement and proof of Proposition \ref{p: equivalence} extends to any space $Y$ for which $H^2(Y;\bZ)$ is a cyclic group with order an odd perfect square and both $Y$ and $-Y$ bound positive definite 4-manifolds with vanishing $H_1$.  A similar but somewhat more complicated conclusion may be drawn without the assumptions on $H^2(Y;\bZ)$.


\section{Mutation of alternating links.}
\label{s: mutation}

\begin{figure}
\includegraphics{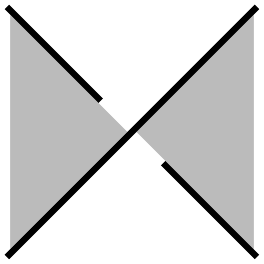}
\caption{Coloring convention for a connected alternating diagram.}
\label{f: crossings}
\end{figure}

Suppose that $D$ is a connected, reduced, alternating diagram of an alternating link $L \subset S^3$.  Color the regions of $D$ in chessboard fashion according to the convention shown in Figure \ref{f: crossings}. 
Let $G(D)$ denote the Tait graph whose vertices correspond to black regions of $D$ and whose edges correspond to crossings where a pair of regions touch.
Denote by $\Lambda(D)$ the lattice of flows on $G(D)$.  This is a positive definite, integral lattice.

The $d$-invariant of a positive definite, integral lattice was defined in \cite[Section 2.4]{greene:2013-2} (see also \cite[Section 2]{greene:2012}).  The definition is intended to mimic the formula of Ozsv\'ath and Szab\'o in \cite[Theorem 3.4]{os:doublecover}.  According to it, the $d$-invariant of $-\Sigma(L)$ is isomorphic to the $d$-invariant of the lattice $\Lambda(D)$ (the notion of isomorphism of $d$-invariants is codified in \cite[Definition 2.2]{greene:2013-2}).  As a consequence, if $D$ and $D'$ are connected, reduced alternating diagrams of a pair of links $L, L'$ for which $\Sigma(L) \approx \Sigma(L')$, then the $d$-invariants of $\Lambda(D)$ and $\Lambda(D')$ are isomorphic.  In other words, \cite[Theorem 3.4]{os:doublecover} implies the following result.

\begin{prop}
\label{p: well-defined}
The isomorphism type of the $d$-invariant of $\Lambda(D)$ is an invariant of the oriented homeomorphism type of $\Sigma(L)$.
\end{prop}

This is the sole input from Heegaard Floer homology used in the proof of the main result of \cite{greene:2013-2} (specifically, see the use of \cite[Theorem 4.7]{greene:2013-2} in the proof of \cite[Theorem 1.1]{greene:2013-2}).  Following Friedl's suggestion, we will derive Proposition \ref{p: well-defined} as an application of Donaldson's theorem, without reference to Heegaard Floer homology.

Before proving Proposition \ref{p: well-defined}, we require some more preparation.  Let $F \subset S^3$ denote the spanning surface for $L$ corresponding to the black regions.  Note that $F$ deformation retracts onto $G(D)$.  Let $D^4$ denote a 4-ball that fills $S^3$ and push $\mathrm{int}(F)$ into $\mathrm{int}(D^4)$ to obtain a properly embedded surface $F'$ that fills $L$.  Let $X(D)$ denote the double-cover of $D^4$ branched along $F'$.  Its boundary is $\Sigma(L)$, and it follows from the work of Gordon and Litherland that $(H_2(X(D));\bZ)$, equipped with its intersection pairing, is isometric to $\Lambda(D)$ \cite[Theorems 1 \& 3]{gl:sig}.  The additional notation in the following proof comes from \cite{greene:2013-2} (cf. \cite{greene:2012}).

\begin{proof}
We must show that if $D$ and $D'$ are connected, reduced, alternating diagrams of a pair of links $L$ and $L'$, and $\Sigma(L) \approx \Sigma(L')$ as oriented manifolds, then $(\cX(\Lambda(D)),d) \approx (\cX(\Lambda(D')),d)$.

Let $\overline{D}$ denote the mirror of the diagram $D$ and $\overline{L}$ the mirror of the link $L$.  We have $\del X(\overline{D}) \approx \Sigma(\overline{L}) \approx -\Sigma(L) \approx -\Sigma(L') \approx - \del X(D')$.  Fix an orientation-reversing homeomorphism $\phi : \del X(\overline{D}) \to \del X(D')$.  Following the discussion in Section \ref{s: rational balls}, we can identify $\cX(\Lambda(D))$ and $\cX(\Lambda(D'))$ with $\Spc(\Sigma(L))$ (note that $H^2(X(D);\bZ)$ does not contain 2-torsion).  The map $\phi$ then descends to a torsor isomorphism
\[
\varphi : \cX(\Lambda(\overline{D})) \to \cX(\Lambda(D')).
\]
We seek to show that $\varphi$ establishes an isomorphism between the $-d$-invariant of $\Lambda(\overline{D})$ and the $d$-invariant of $\Lambda(D')$: that is, $d(\varphi(x)) = -d(x)$ for all $x \in \cX(\Lambda(\overline{D}))$.

Form the closed, oriented, smooth, definite 4-manifold $Z = X(\overline{D}) \cup_\phi X(D')$.  By Donaldson's theorem, $(H_2(Z;\bZ),Q_Z)$ is isometric to the Euclidean lattice $\bZ^n$, where $n = \rk(\Lambda(\overline{D}))+\rk(\Lambda(D'))$.
Given a class $x \in \cX(\Lambda(\overline{D}))$, select $\chi \in \Short(\Lambda(\overline{D}))$ with $[\chi] = x$ and $\chi' \in \Short(\Lambda(D'))$ with $[\chi'] = \varphi(x)$.  Then $\chi + \chi' \in \Char(H_2(Z)) = \Char(\bZ^n)$, so $|\chi + \chi'| \ge n$.  It follows that
\begin{eqnarray*}
d(x) + d(\varphi(x)) &=& \frac 14(|\chi| -  \rk(\Lambda(\overline{D}))) + \frac 14(|\chi'| -  \rk(\Lambda(D'))) \\
&\ge& 0, \quad \forall x \in \cX(\Lambda(\overline{D})).
\end{eqnarray*}

Similarly, consideration of the pair $(D,\overline{D}')$ yields a torsor isomorphism
\[
\psi : \cX(\Lambda(D)) \to \cX(\Lambda(\overline{D}'))
\]
with the property that
\[
d(y) + d(\psi(y)) \ge 0,  \quad \forall y \in \cX(\Lambda(D)).
\]
Adding the two collections of inequalities above over all $x$ and $y$ shows that the sum of the total $d$-invariants of $\Lambda(D), \Lambda(D'), \Lambda(\overline{D})$, and $\Lambda(\overline{D}')$ is non-negative.   

On the other hand, since $G(D)$ and $G(\overline{D})$ are planar dual to one another, 
there exists an isomorphism
\[
(\cX(\Lambda(D)),d) \approx (\cX(\Lambda(\overline{D})),-d)
\]
by \cite[Corollary 3.4]{greene:2013-2}.  Similarly, there exists an isomorphism between $(\cX(\Lambda(D')),d)$ and $(\cX(\Lambda(\overline{D}')),-d)$.
Thus, the sum of the total $d$-invariants of the four lattices vanishes, and it follows that $d(x) + d(\varphi(x)) = 0, \forall x \in \cX(\Lambda(D))$.  Therefore, $\varphi$ establishes an isomorphism
\[
(\cX(\Lambda(\overline{D})),-d) \approx (\cX(\Lambda(D')),d).
\]
Combining the last two indented expressions, it follows that the $d$-invariants of $\Lambda(D)$ and $\Lambda(D')$ are isomorphic, as desired.
\end{proof}

It is intriguing to consider the role of Donaldson's theorem in the proof of Proposition \ref{p: well-defined}.  Is it possible to supply a more direct topological argument that $(H_2(Z,\bZ),Q_Z) \approx \bZ^n$ for the specific type of 4-manifold $Z$ appearing in it?

Lastly, we point out another line of argument for establishing \cite[Theorem 3.4]{os:doublecover}.  It is sufficient to show that the 4-manifold $X(D)$ is sharp.  
 By the combinatorial argument \cite[Corollary 3.4]{greene:2013-2}, the $d$-invariants of $\Lambda(D)$ and $\Lambda(\overline{D})$ are negative one another.  The first gives a lower bound on the $d$-invariant of $\Sigma(L)$, while the second gives a lower bound on the $d$-invariant of $\Sigma(\overline{L})$, both by \cite[Theorem 9.6]{os:absgr}.  The $d$-invariants of these manifolds are negative of one another by orientation-reversal.  It follows that both lower bounds are sharp, so $X(D)$ is sharp.

\bibliographystyle{myalpha}
\bibliography{/Users/JoshuaGreene/Dropbox/Papers/References}

\providecommand{\bysame}{\leavevmode\hbox to3em{\hrulefill}\thinspace}
\providecommand{\MR}{\relax\ifhmode\unskip\space\fi MR }
\providecommand{\MRhref}[2]{%
  \href{http://www.ams.org/mathscinet-getitem?mr=#1}{#2}
}
\providecommand{\href}[2]{#2}
\begin{thebibliography}{GRS08}

\bibitem[CG86]{cg:cobordism}
A.~J. Casson and C.~McA. Gordon, \emph{Cobordism of classical knots}, \`{A} la
  recherche de la topologie perdue, Progr. Math., vol.~62, Birkh\"auser Boston,
  Boston, MA, 1986, pp.~181--199.

\bibitem[CG88]{cg:1988}
T.~D. Cochran and R.~E. Gompf, \emph{Applications of {D}onaldson's theorems to
  classical knot concordance, homology {$3$}-spheres and property {$P$}},
  Topology \textbf{27} (1988), no.~4, 495--512.

\bibitem[Don87]{d:thma}
S.~K. Donaldson, \emph{The orientation of {Y}ang-{M}ills moduli spaces and
  {$4$}-manifold topology}, J. Differential Geom. \textbf{26} (1987), no.~3,
  397--428.

\bibitem[Elk95]{elkies}
N.~D. Elkies, \emph{A characterization of the {$\mathbb{Z}^n$} lattice}, Math.
  Res. Lett. \textbf{2} (1995), no.~3, 321--326.

\bibitem[FQ90]{fq:book}
M.~H. Freedman and F.~Quinn, \emph{Topology of 4-manifolds}, Princeton
  Mathematical Series, vol.~39, Princeton University Press, Princeton, NJ,
  1990. \MR{1201584 (94b:57021)}

\bibitem[GJ11]{greenejabuka:2011}
J.~E. Greene and S.~Jabuka, \emph{The slice-ribbon conjecture for 3-stranded
  pretzel knots}, Amer. J. Math. \textbf{133} (2011), no.~3, 555--580.

\bibitem[GL78]{gl:sig}
C.~McA. Gordon and R.~A. Litherland, \emph{On the signature of a link}, Invent.
  Math. \textbf{47} (1978), no.~1, 53--69.

\bibitem[Gre12]{greene:2012}
J.~E. Greene, \emph{Conway mutation and alternating links}, Proceedings of the
  {G}\"okova {G}eometry-{T}opology {C}onference 2011, Int. Press, Somerville,
  MA, 2012, pp.~31--41.

\bibitem[Gre13]{greene:2013-2}
\bysame, \emph{Lattices, graphs, and {C}onway mutation}, Invent. Math.
  \textbf{192} (2013), no.~3, 717--750.

\bibitem[GRS08]{grs:conc}
J.~E. Grigsby, D.~Ruberman, and S.~Strle, \emph{Knot concordance and {H}eegaard
  {F}loer homology invariants in branched covers}, Geom. Topol. \textbf{12}
  (2008), no.~4, 2249--2275.

\bibitem[GS99]{gs:4mflds}
R.~E. Gompf and A.~I. Stipsicz, \emph{{$4$}-manifolds and {K}irby calculus},
  Graduate Studies in Mathematics, vol.~20, American Mathematical Society,
  Providence, RI, 1999.

\bibitem[JN07]{jn:conc}
S.~Jabuka and S.~Naik, \emph{Order in the concordance group and {H}eegaard
  {F}loer homology}, Geom. Topol. \textbf{11} (2007), 979--994.

\bibitem[Lec12]{ana:thesis}
A.~G. Lecuona, \emph{On the slice-ribbon conjecture for {M}ontesinos knots},
  Trans. Amer. Math. Soc. \textbf{364} (2012), no.~1, 233--285.

\bibitem[Lis07a]{lisca:lens1}
P.~Lisca, \emph{Lens spaces, rational balls and the ribbon conjecture}, Geom.
  Topol. \textbf{11} (2007), 429--472.

\bibitem[Lis07b]{lisca:lens2}
\bysame, \emph{Sums of lens spaces bounding rational balls}, Algebr. Geom.
  Topol. \textbf{7} (2007), 2141--2164.

\bibitem[OS03]{os:plumbed}
P.~Ozsv{\'a}th and Z.~Szab{\'o}, \emph{On the {F}loer homology of plumbed
  three-manifolds}, Geom. Topol. \textbf{7} (2003), 185--224 (electronic).

\bibitem[OS12]{owensstrle:2012}
Brendan Owens and Sa{\v{s}}o Strle, \emph{A characterization of the {$\Bbb
  Z^n\oplus\Bbb Z(\delta)$} lattice and definite nonunimodular intersection
  forms}, Amer. J. Math. \textbf{134} (2012), no.~4, 891--913.

\bibitem[OSz03]{os:absgr}
P.~Ozsv{\'a}th and Z.~Szab{\'o}, \emph{Absolutely graded {F}loer homologies and
  intersection forms for four-manifolds with boundary}, Adv. Math. \textbf{173}
  (2003), no.~2, 179--261.

\bibitem[OSz05]{os:doublecover}
\bysame, \emph{On the {H}eegaard {F}loer homology of branched double-covers},
  Adv. Math. \textbf{194} (2005), no.~1, 1--33.

\end{thebibliography}

\end{document}